\documentclass[10pt,leqno]{amsart}
\usepackage{graphicx}
\usepackage{indentfirst,csquotes}

\usepackage{graphicx} % Required for inserting images
\usepackage{amsmath} % 提供基本数学公式支持
\usepackage{amssymb} % 提供数学符号
\usepackage{amsfonts} % 提供数学字体
\usepackage{stackrel} % 支持上下标叠加
\usepackage{amsthm} % proof环境
\usepackage{physics} % 取值宏包
\usepackage{hyperref} % 邮箱超链接
\usepackage{lipsum}

\numberwithin{equation}{section}
\newtheorem{definition}{{Definition}}[section]
\newtheorem{theorem}[definition]{{Theorem}}

\newtheorem{lemma}[definition]{{Lemma}}

\newtheorem{remark}{{Remark}}

\hypersetup{ colorlinks=true, linkcolor=black, filecolor=black, urlcolor=black }

\begin{document}
\title{Limiting Distribution and Rate of Convergence for  GL(3) Fourier Coefficients}
\author{Zongqi Yu}
\address{Data Science Institute, Shandong University, Jinan 250100, China}
\email{zongqi.yu@mail.sdu.edu.cn}

\let\thefootnote\relax

\footnote{\textit{Date.} April, 2026}



\begin{abstract}
    In a work of Heath-Brown, it is proved that in the Pilz divisor problem, the normalized error term $\Delta_3(x)$ has a distribution function. In this paper, we prove an analogue of this result in the setting of $\operatorname{GL}(3)$. For a given self-dual $\operatorname{GL}(3)$ Hecke--Maass cusp form $f$ with normalized Fourier coefficients $A_f(n,m)$, let $\Delta_f(x)=\sum_{n\leqslant x}A_f(n,1)$. We show that the function $x^{-1/3}\Delta_f(x)$ has a distribution function and we obtain a quantitative rate of convergence for the limiting distribution.
\end{abstract}
\maketitle

\section{Introduction}
In analytic number theory, the classical divisor problem concerns the asymptotic behavior of the sum
\[
\Delta(x) = \sum_{n \leqslant x} d(n) - x\bigl(\log x + 2\gamma - 1\bigr),
\]
where $d(n)$ is the divisor function and $\gamma$ is Euler's constant. Wintner \cite{MR3645} explicitly gives a value distribution of the normalized $\Delta(x)$. 

The development and application of probabilistic methods in analytic number theory have greatly enhanced our understanding of the distribution of such error terms. In 1992, Heath-Brown \cite{MR1159354} proved that  the normalized error term $F(t)= t^{-1/4}\Delta(t)$ has a limiting distribution. For any interval $I \subset \mathbb{R}$, he showed that
\begin{equation}\label{eq1.1}
\lim_{T \to \infty} \frac{1}{T} \operatorname{meas} \bigl\{ t \in [T, 2T] : F(t) \in I \bigr\} = \int_I g(u) du,
\end{equation}
where the limiting density $g$ belongs to $C^{\infty}(\mathbb{R})$ and can be extended to an entire function on $\mathbb{C}$. Moreover, in that paper, Heath-Brown also proved that similar results hold for the error term $\Delta_3(x)=\sum_{n\leq x}d_3(n)-\underset{s=1}{\operatorname*{\operatorname*{Res}}}\left(\frac{\zeta^3(s)x^s}{s}\right)$, where $d_3(n)$ is the triple divisor function, providing that $x^{1/4}\Delta(x)$ is replaced by $x^{1/3}\Delta_3(x)$.

The method of Heath-Brown has many applications. In 1993, Bleher et al. \cite{MR1224087} proved that the error term for the number of lattice points inside a shifted circle has a limiting distribution. After that,  Bleher also conducted extensive related work; see reference \cite{MR1224102,MR1334401}. Using similar methods, Cai \cite{MR1633874} proved analogous results on $\operatorname{GL}(2)$. 

A profound insight of \cite{MR1159354} is that such a distribution can be obtained by approximating probability models. For a measurable function $h$ on $[T, 2T]$ and an event $\mathcal{B}$, we define
\[
\mathbb{P}_T\big(h(t) \in \mathcal{B}\big) := \frac{1}{T} \operatorname{meas}\{ t \in [T, 2T] : h(t) \in \mathcal{B} \}.
\]
Let \(\{\mathbb{X}_n\}_{n \text{ square-free}}\) be a sequence of independent random variables, each uniformly distributed on \([0,1]\). Consider the random trigonometric series
\begin{equation}\label{eq1.2}
F_{\mathbb{X}} := \frac{1}{\pi \sqrt{2}} \sum_{\substack{n=1\\ n \text{ square-free}}}^{\infty} \frac{\mu(n)^2}{n^{3/4}} \sum_{r=1}^{\infty} \frac{d(n r^2)}{r^{3/2}} \cos\left(2\pi r \mathbb{X}_n - \frac{\pi}{4}\right).
\end{equation}
 Heath-Brown's result is that the density $g$ in \eqref{eq1.1} is precisely the probability density function of the random variable $F_{\mathbb{X}}$. Furthermore,
\begin{equation}\label{eq1.3}
\lim_{T \to \infty} \mathbb{P}_T\big(F(t) \le u\big) = \mathbb{P}\big(F_{\mathbb{X}} \le u\big) \qquad \text{for all } u \in \mathbb{R}.
\end{equation}
The motivation for the model \eqref{eq1.2} comes from the analytic structure of $\Delta(t)$ itself. Voronoĭ's summation formula  yields that uniformly for $t \in [T, 2T]$ and any fixed $\varepsilon > 0$,
\[
F(t) = \frac{1}{\pi \sqrt{2}} \sum_{n \le T} \frac{d(n)}{n^{3/4}} \cos\left(4\pi\sqrt{nt} - \frac{\pi}{4}\right) + O(T^{\varepsilon}).
\]
Writing $n = m r^2$ with $m$ is square-free, we have
\[
F(t) = \frac{1}{\pi \sqrt{2}} \sum_{m \le T} \frac{\mu(m)^2}{m^{3/4}} \sum_{r \le \sqrt{T/m}} \frac{d(m r^2)}{r^{3/2}} \cos\left(4\pi r\sqrt{m t} - \frac{\pi}{4}\right) + O(T^{\varepsilon}).
\]

Lau and Tsang \cite{MR2111933}  provided explicit formulas for the moments of the probability density function of a  class of functions, and proved that the density function of the remainder term of the divisor function is asymmetric.

 Lau \cite{MR1916279}  established a quantitative bound on the discrepancy between the distribution function of $F(t)$ and that of the random model $F_\mathbb{X}$. More specifically, he proved that 
\[
D_F(T) := \sup_{u \in \mathbb{R}} \left| \mathbb{P}_T(F(t) \leqslant u) - \mathbb{P}(F_{\mathbb{X}} \leqslant u) \right|\ll \frac{(\log\log \log T)^{3/4}}{(\log \log T)^{1/8}}.
\]

Based on ideas from the work \cite{MR4041109} of Lamzouri, Lester and Radziwiłł, Lamzouri \cite{MR4855345} improved the estimate for $D_F(T)$. He proved that
\[
D_F(T) \ll \frac{(\log\log \log  T)^{9/4}}{(\log \log T)^{3/4}}.
\]
Furthermore, by a result of Lau \cite{MR1862056},  Lamzouri also provided the distribution of large values of $\Delta(t)$ in that paper. More specifically, he proved that there exist positive constants $b_1$, $b_2$, $b_3$, $b_4$ such that for all real numbers $V$ in the range $b_1\leqslant V\leqslant b_2(\log \log T)^{1/4}(\log\log \log T)^{2^{4/3}-7/4}$, one has
\[
\exp(-b_3V^4(\log V)^{-3(2^{4/3}-1)})\leqslant \mathbb{P}_T(F(t)>V)\leqslant \exp(-b_4V^4(\log V)^{-3(2^{4/3}-1)}).
\]
Moreover, the same bounds hold for $\mathbb{P}_T(F(t)<V)$, in the same range of $V$. 

In this paper, we are not confined to the classical divisor function, but rather  a more general situation: arithmetic functions defined by the Fourier coefficients of automorphic forms on $\operatorname{GL}(3)$. 
Specifically, let $f$ be a fixed self-dual Hecke--Maass cusp form for $\operatorname{GL}(3)$. Let $A_f(n,m)$ be its normalized Fourier coefficients and let
\[
\Delta_f(x):=\sum_{n\leqslant x}A_f(n,1).
\]
We prove that the normalized function $t^{-1/3} \Delta_f\left(t\right)$ has a  distribution function $g(\alpha)$. More precisely, we have
\begin{theorem}\label{th1.2}
    Let $f$ be a self-dual Hecke--Maass cusp form for $\operatorname{GL}(3)$. The function  $F_f(t)=t^{-1/3}\Delta_f\left(t\right)$ has a distribution function $g(\alpha)$,  in the sense that for any interval I, we have
\[
X^{-1} \operatorname{maes}\left\{t \in[1, X]: F_f(t) \in I\right\} \rightarrow \int_I g(\alpha) d \alpha,
\]
as $X \rightarrow \infty$. The function $g(\alpha)$ and its derivatives satisfy the growth condition
\[
\frac{d^k}{d \alpha^k} g(\alpha) \ll_{A, k}(1+|\alpha|)^{-A},
\]
for $k=0,1,2, \ldots$ and any constant A. Moreover, $g(\alpha)$ can be extended to an entire function on $\mathbb{C}$.
\end{theorem}

We use the method discussed in Heath-Brown’s work \cite{MR1159354} for estimating the divisor function to prove Theorem \ref{th1.2}. 
In the proof, we do not assume the validity of the Ramanujan conjecture, but rather replace the pointwise estimates in the original method with mean value estimates of Fourier coefficients.

In order to prove  Theorem \ref{th1.2}, we need an estimate for the mean square of the partial sums of the Fourier coefficients. Specifically, we have 
\begin{theorem}\label{th1.3}
   Let $f$ be a self-dual Hecke--Maass cusp form for $\operatorname{GL}(3)$. For any positive $\varepsilon$, we have
    \[
    \int_{0}^{x}\Delta_f\left(y\right)^2dy=\frac{1}{10\pi^2}\sum_{n=1}^\infty\frac{A_f\left(1,n\right)^2}{n^{\frac{4}{3}}}x^{\frac{5}{3}}+O\left(x^{\frac{18}{11}+\varepsilon}\right).
    \]
\end{theorem}
The key step in proving Theorem \ref{th1.3} is that we found a $\sigma<2/3$ such that
\[
\int_{1}^T \left|L(\sigma+it,f)\right|^2dt\ll T^{1+\varepsilon}.
\]
which is made possible thanks to the work \cite{MR4945119} of Pal.

Next, we use the method used by Lamzouri \cite{MR4041109} in treating the divisor function. However, unlike the case of the divisor function, we use mean value estimates of the Fourier coefficients in place of the pointwise estimates. We establish a quantitative  rate of convergence for the limiting distribution in Theorem \ref{th1.2}.
\begin{theorem}\label{th1.4} 
   Let $\{\mathbb{X}_n\}_{n \text{cube-free}}$ be a sequence of independent random variables, each uniformly distributed on \([0,1]\).  For a fixed self-dual Hecke--Maass cusp form $f$ of $\operatorname{GL}(3)$, we have
    \[
\sup_{u \in \mathbb{R}} \left| \mathbb{P}_T(F_f(t) \leqslant u) - \mathbb{P}(F_{f,\mathbb{X}} \leqslant u) \right|\ll \frac{\left(\log\log\log   T\right)^\frac{5}{3}}{\left(\log\log  T\right)^{\frac{1}{3}}},
\]
where \begin{equation}\label{eq1.5}
F_{f,\mathbb{X}} := \frac{1}{\pi \sqrt{3}} \sum_{\substack{n=1}}^{\infty} \frac{\varepsilon(n)}{n^{2/3}} \sum_{r=1}^{\infty} \frac{A_f(1,n r^3)}{r^{2}} \cos\left(6\pi r \mathbb{X}_n\right),
\end{equation}
with  $\varepsilon(n)=1$ if $n$ is cube-free and $\varepsilon(n)=0$ otherwise.
\end{theorem}

Finally, following the approach of Lau \cite{MR1862056} and Lamzouri \cite{MR4041109} we determine the upper and lower bounds of the Laplace transform of $F_{f,\mathbb{X}}$ and obtain the distribution of large values of $\Delta_f(t)$. Due to the lack of strong pointwise estimates for the Fourier coefficients and the inability to apply mean-value estimates, we  obtain the following weaker result.
\begin{theorem}\label{th1.5}
     For a fixed self-dual Hecke--Maass cusp form $f$ of $\operatorname{GL}(3)$ and any positive $\varepsilon_1$, $\varepsilon_2$. There exist positive constants  $b_{f,1}$, $b_{f,2}$, $b_{f,3}$, $b_{f.3}$ such that for all  real numbers $V$ in the range $b_{f,1}\leqslant V \leqslant b_{f,2} (\log \log T)^{1/21}(\log \log \log T)^{-5/21}$, we have
    \[
    \exp(-b_{f,3}V^{35/2+\varepsilon_1})\leqslant \mathbb{P}_T(F_f(t)>T) \leqslant \exp(-b_{f,4} V^{5/3-\varepsilon_2}).
    \]
    Moreover, the same bounds hold for $\mathbb{P}_T(F_f(t)<V)$, in the same range of $V$.
\end{theorem}

\begin{remark}
Throughout the paper, we require that $f$ is a self-dual Hecke--Maass cusp form. This is because in the proof of Theorem \ref{th1.2}, such a condition is needed to ensure that
\[
a_n(t) = \frac{\varepsilon(n)}{\pi \sqrt{3}} \sum_{r=1}^{\infty} \frac{A_f(1, n r^3)}{(n r^3)^{2/3}} \cos(6\pi r t)
\] 
is well-defined, i.e. the series in the expression is convergent. However, in the case where $f$ is not necessarily self-dual, the convergence of the corresponding series remains unknown. If it can be proven that the corresponding series converges when $f$ is a more general Hecke--Maass  cusp form, then all results in this paper can be extended to the corresponding cases.
\end{remark}

\section{Preliminaries}
Let $f$ be a self-dual Hecke--Maass cusp form of type $v=(v_f,v_f)$ for $\operatorname{GL}(3)$  with  the normalized Fourier coefficients $A_f(n,m)$. We have the conjugation relation $A_f(n,m)=\overline{A_f(m,n)}=A_f(m,n)$, see \cite[Theorem 9.3.11]{MR2254662} and the  following multiplicativity relations,
\[
\begin{aligned}
&A_f\left(n_1 n_2, m_{1} m_{2}\right)=A_f\left(n_1, m_{2}\right)  A\left(n_2, m_{2}\right), \quad \text{if }  \left(n_1 m_{1}, n_2  m_{2}\right)=1,
\\
    &A_f(m, 1) A_f\left(m_1, m_{2}\right)=\sum_{\substack{ c_1c_2c_3=m \\
c_1\left|m_1, c_2\right| m_2}} A_f\left(\frac{m_1 c_3}{c_1}, \frac{m_2 c_1}{c_2}\right) ,\\
    &A_f(m_1, 1) A_f\left(1, m_{2}\right)=\sum_{d\mid (m_1,m_2)} A_f\left(\frac{m_1}{d}, \frac{m_2}{d}\right) .
\end{aligned}
\]
Now, we define the $L$-function attached to $f$ by
\[
L(s, f)=\sum_{n=1}^{\infty} \frac{A_f(n, 1)}{n^s},
\]
for $\operatorname{Re}(s)>1$, and by analytic continuation for all $s$ in the complex plane. Next, we define the gamma factor
\[
\gamma(s, f)=\pi^{-\frac{3\pi}{2}} \Gamma\left(\frac{s+1-3v_f}{2}\right) \Gamma\left(\frac{s}{2}\right) \Gamma\left(\frac{s-1+3v_f}{2}\right) .
\]
Then we have the  functional equation for $L(s, f)$, i.e.,
\[
\gamma(s, f) L(s, f)=\gamma(1-s, \tilde{f}) L(1-s, \tilde{f}) .
\]
Kim and Sarnak \cite[Appendix 2]{MR1937203} have proven the following bounds towards the Ramanujan Conjecture
\[
\begin{aligned}
& \left|A_f(n, 1)\right| \ll n^{\frac{5}{14}+\epsilon}.
\end{aligned}
\]
On average, by the Rankin–Selberg estimate \cite{MR4211098}, we have
\[
\begin{aligned}
& \sum_{m^2n \leqslant N}|A_f(m,n)|^2 \ll_f N.
\end{aligned}
\]
By Cauchy--Schwarz inequality, one derives that
\begin{equation}\label{eq2.1}
    \sum_{n \leqslant N}|A_f(n, 1)| \ll_f N .
\end{equation}

\section{The proof of Theorem \ref{th1.3}}

Before proving Theorem \ref{th1.2}, we need to  prove Theorem \ref{th1.3}. First, following Atkinson \cite{MR6190} and  Tong's method \cite{MR98718} in handling $d_3(n)$ (the triple divisor function), we will obtain following results.
\begin{lemma}\label{le3.1}
    Let $f$ be a fixed self-dual Hecke--Maass cusp form  for $\operatorname{GL}(3)$. For $\frac{1}{2}<\alpha<\frac{2}{3}$ and  any positive $\varepsilon$, we have
    \[
    \Delta_f(x)=\frac{x^\frac{1}{3}}{\pi \sqrt{3}}\sum_{n\leqslant X}\frac{A_f(1,n)}{n^{\frac{2}{3}}}\cos{6\pi (nx)^{\frac{1}{3}}}+O(x^{1-\alpha+\varepsilon}),
    \]
    where $X$ denotes $x^{3\alpha-1}/8\pi^3$. 
\end{lemma}
\begin{proof}
    The proof is similar to \cite{MR6190} , where we need to use the functional equation
    \[
    \gamma(s, f) L(s, f)=\gamma(1-s, \tilde{f}) L(1-s, \tilde{f}) .
    \]
    and $\sum_{n \leqslant N}|A_f(n, 1)|^2 \ll_f N$.
\end{proof}
\begin{lemma}\label{pro4.1}
Let $f$ be a fixed self-dual Hecke--Maass cusp form  for $\operatorname{GL}(3)$ and $\sigma_1$ be the greatest lower bound for which
\[
\int_{-T}^{T}\left|L\left(\sigma_1+it,f\right)\right|^2dt\ll T^{1+\varepsilon}
\]
holds.
    If $\sigma_1\leqslant\frac{2}{3}$, then for any positive $\varepsilon$ we have
    \[
    \int_{0}^{x}\Delta_f\left(y\right)^2dy=\frac{1}{10\pi^2}\sum_{n=1}^\infty\frac{A_f\left(1,n\right)^2}{n^{\frac{4}{3}}}x^{\frac{5}{3}}+O\left(x^{2-\frac{3-4\sigma_1}{6\left(1-\sigma_1\right)-1}+\varepsilon}\right).
    \]
\end{lemma}

\begin{proof}
    Combining Lemma \ref{le3.1} and \cite[Theorem 1]{MR98718}, the proof follows.
\end{proof}

Now we only need to find a $\sigma_1<\frac{2}{3}$  to prove Theorem \ref{th1.3}. For $\frac{1}{2}<\sigma<1$, we define $m(\sigma)$ be the supremum of all numbers  $m$ such that
\begin{equation}\label{def m}
    \int_1^T\left|L(s,f)\right|dt^m\ll T^{1+\varepsilon}.
\end{equation}
Thus, we transform the problem into finding the smallest $\sigma$ such that $m(\sigma)\geqslant 2$.
\begin{lemma}\label{le4.10}
    Let $k\geqslant 1$ be a fixed integer and let $\frac{T}{2}\leqslant t \leqslant 2T$. Then for $\frac{1}{2}<  \sigma <1$ fixed we have 
    \begin{equation}\label{eq4.41}
        |L(\sigma+it ,f)|^k\ll 1+\log T \int_{-\log^2T}^{\log^2T}\left| L\left(\sigma-\frac{1}{\log T}+it+iv,f\right)\right|^ke^{-|v|}dv.
    \end{equation}
    and 
    \[
    \left|L(\frac{1}{2}+it,f)\right|^k\ll \log T\left(1+\int_{-\log^2T}^{\log ^2 T}\left|L(\frac{1}{2}+it+iv,f)\right|^ke^{-|v|}dv\right).
    \]
\end{lemma}
\begin{proof}
    The proof is similar to \cite[Lemma 7.1]{MR792089}.
\end{proof}

\begin{lemma}\label{le4.11}
    For $\frac{1}{2}<\sigma<1$. 
    Let $R$ be  a positive integer  such that
    \[
    T \leqslant t_r \leqslant 2T, \qquad r=1,\dots,R,
    \]
    and
    \[
    |t_r - t_s| \geqslant \log^{C} T, \qquad 1 \leqslant r \neq s \leqslant R,
    \]
    where $C \geqslant 0$ is a fixed constant. Then the estimate 
    \[
    \int_1^T |L(\sigma+it,f)|^{m(\sigma)}dt\ll T^{1+\varepsilon}
    \]
    is equivalent to
    \[
    \sum_{r\leqslant R}|L(\sigma+it_r,f)|^{m(\sigma)}\ll T^{1+\varepsilon}.
    \]
\end{lemma}
\begin{proof}
    The proof is similar to \cite[Section 8.1]{MR792089}.
\end{proof}

\begin{lemma}\label{le4.12}
    Suppose that $\frac{1}{2}<\sigma<1$ is fixed and
    \[
    |L(\sigma+it_r,f)|\geqslant V \geqslant T^\varepsilon.
    \]
    Then 
    \begin{equation}\label{eq441}
        R\ll T^{1+\varepsilon}V^{-m(\sigma)}
    \end{equation}
    is equivalent to
    \begin{equation}\label{eq442}
        \sum_{r\leqslant R}|L(\sigma+it_r,f)|^{m(\sigma)}\ll T^{1+\varepsilon},
    \end{equation}
    where $t_r$ is defined as Lemma \ref{le4.11}.
\end{lemma}
\begin{proof}
    Suppose that \eqref{eq442} holds and let $t_{V,1},\cdots,t_{V,R_1}$ be the subset of $\{t_r\}$ such that 
    \[
    \left|L(\sigma+it_r,f)\right|\geqslant V \quad (j=1,\cdots,R_1).
    \]
    Then \eqref{eq442} implies 
    \[
    R_1V^{m(\sigma)}\ll \sum_{r\leqslant R_1} \left|L(\sigma+it_r,f)\right|^{m(\sigma)} \ll T^{1+\varepsilon},
    \]
    and consequently \eqref{eq441} holds with $R_1=R$. Conversely, let \eqref{eq441} holds and denote by $t_{V,1},\cdots,t_{V,R(V)}$ those of the points $t_1,\cdots,t_R$ for which
    \[
    V\leqslant \left|L(\sigma+t_{V,j},f)\right|\ll 2V \quad (j=1,\cdots,R(V)).
    \]
    There are $O\left(\log T\right)$ choices for each $V$, then we have
    \[
    \begin{aligned}
        \sum_{r\leqslant R} \left|L(\sigma+it_r,f)\right|^{m(\sigma)}&\ll RT^\varepsilon+\sum_V\sum_{j\leqslant R(V)}\left(2V\right)^{m(\sigma)}\\
        &\ll RT^\varepsilon+\sum_VT^{1+\varepsilon}\ll T^{1+\varepsilon}.
    \end{aligned}
    \]
    This completes the proof.
\end{proof}
\begin{lemma}\label{le4.13}
   Let $R$ be  a positive integer  such that
    \[
    T \leqslant t_r \leqslant 2T, \qquad r=1,\dots,R,
    \]
    and
    \[
    |t_r - t_s| \geqslant \log^{4} T, \qquad 1 \leqslant r \neq s \leqslant R.
    \]
     If
    \[
    T^\varepsilon<V \leqslant\left|\sum_{M \leqslant n \leqslant 2 M} a(n) n^{\sigma-i t_r}\right|,
    \]
    where $\sum_{n\leqslant M}|a(n)| \ll M^{1+\varepsilon}$ and $1 \ll M \ll T^C(C>0)$. Then
    \begin{equation}\label{eq4.43}
        R \ll T^\varepsilon\left(M^{2-2 \sigma} V^{-2}+T V^{-l(\sigma)}\right) .
    \end{equation}
    Here
    \[
    l(\sigma)= \begin{cases}\frac{2}{3-4 \sigma} & \text { if } \frac{1}{2}<\sigma \leq \frac{2}{3}, \\ \frac{10}{7-8 \sigma} & \text { if } \frac{2}{3}<\sigma \leqslant \frac{11}{14}, \\ \frac{34}{15-16 \sigma} & \text { if } \frac{11}{14}<\sigma \leqslant \frac{13}{15}, \\ \frac{98}{31-32 \sigma} & \text { if } \frac{13}{15}<\sigma \leqslant \frac{57}{62}, \\ \frac{5}{1-\sigma} & \text { if } \frac{57}{62}<\sigma \leqslant 1-\varepsilon .\end{cases}
    \]
\end{lemma}

\begin{proof}
    We can get this lemma by following a similar argument to \cite[Lemma 8.2]{MR792089} replacing $a(n)\ll M^\varepsilon$ by $\sum_{n\leqslant M}|a(n)|\ll M^{1+\varepsilon}$.
\end{proof}
\begin{lemma}\label{le4.14}
     Let $f$ be a Hecke--Maass cusp form from for $\operatorname{GL}(3)$ and $\frac{1}{2}\leqslant \sigma \leqslant 1$ be fixed, we have
    \[
    \left|L(\sigma+it ,f)\right|\ll \left|t\right|^{\frac{3}{2}(1-\sigma)+\varepsilon}.
    \]
\end{lemma}
\begin{proof}
    This result follows directly from the convexity bound $L(1/2+it ,f)\ll _f (1+\left|t\right|)^{3/4+\varepsilon}$ and the  Phragm\'en--Lindel\"of theorem \cite[Theorem 8.2.1]{MR2254662}.
\end{proof}
\begin{lemma}\label{le4.15}
     Let $f$ be a Hecke--Maass cusp form from for $\operatorname{GL}(3)$. Then we have
    \[
    \int_{T}^{2T}\left|L(\frac{1}{2}+it,f)\right|^2 dt\ll  T^{\frac{3}{2}-\frac{3}{32}+\varepsilon}.
    \]
\end{lemma}
\begin{proof}
    See \cite[Theorem 1.0.1]{MR4945119} for details.
\end{proof}

\begin{remark}
    We can refine the two lemmas mentioned above respectively using the subconvexity bound $L(1/2+it,f)\ll_f(1+\left|t\right|)^{3/4-3/40+\varepsilon}$ \cite[Theorem 1.1]{MR4270876} and the work of Dasgupta, Leung and Young in the preprint \cite{arXiv:2407.06962}, leading to a minor improvement in the following lemma.
\end{remark}

\begin{lemma}
    For $\frac{13}{20}\leqslant \sigma <1$, let $m(\sigma)$ be defined by \eqref{def m}. Then we have
    \[
    m(\sigma)\geqslant \frac{2}{3}\frac{3-2\sigma+\frac{3(1-\sigma)}{8}}{(3-2\sigma)(1-\sigma)}.
    \]
\end{lemma}
\begin{proof}
By Lemma \ref{le4.11} and Lemma \ref{le4.12}, it suffices to prove
\[
R \ll T^{1+\varepsilon} V^{-m(\sigma)} .
\]

We begin with a standard application of the Mellin transform. For a parameter $Y > 0$ (to be chosen later) and $s = \sigma + i t_r$ with $\frac{13}{20} \le \sigma < 1$ and $t_r$ satisfying the conditions of Lemma \ref{le4.11}, we have
\[
\sum_{n \ge 1} \left( \sum_{n = n_1 n_2} A_f(n_1, 1) A_f(n_2, 1) \right) e^{-n / Y} n^{-s}
= \frac{1}{2 \pi i} \int_{(2)} Y^w \Gamma(w) L(s+w, f)^2   dw.
\]
Shifting the contour to $\Re(w) = \frac{1}{2} - \sigma$ and applying the residue theorem yields
\[
\begin{aligned}
\sum_{n \le Y } &\left( \sum_{n = n_1 n_2} A_f(n_1, 1) A_f(n_2, 1) \right) e^{-n / Y} n^{-s} + o(1) \\
&= L(s, f)^2 + \frac{1}{2 \pi i} \int_{\left(\frac{1}{2} - \sigma\right)} Y^w \Gamma(w) L(s+w, f)^2   dw.
\end{aligned}
\]
 Applying Stirling's formula to estimate the gamma factors and bounding the integral, we obtain
\[
\begin{aligned}
L(s, f)^2 \ll & \; 1 + \left| \sum_{n \le Y } \left( \sum_{n = n_1 n_2} A_f(n_1, 1) A_f(n_2, 1) \right) e^{-n / Y} n^{-s} \right| \\
& + \left| \int_{-\log^2 T}^{\log^2 T} Y^{\frac{1}{2} - \sigma} L\left(\tfrac{1}{2} + i t_r + i v, f\right)^2 e^{-|v|}   dv \right|.
\end{aligned}
\]
Consequently, at least one of the following two inequalities must hold:
\begin{equation}\label{eq4.51}
    V^2 \ll \log Y \max_{M \le \frac{1}{2} Y } \left| \sum_{M \le n \le 2M} \left( \sum_{n = n_1 n_2} A_f(n_1, 1) A_f(n_2, 1) \right) e^{-n / Y} n^{-s} \right|,
\end{equation}
or
\begin{equation}\label{eq4.52}
    V^2 \ll Y^{\frac{1}{2} - \sigma} \left| L\left( \tfrac{1}{2} + i t_r', f \right) \right|^2,
\end{equation}
where
\[
\left| L\left( \tfrac{1}{2} + i t_r', f \right) \right| := \max_{-\log^2 T \le v \le \log^2 T} \left| L\left( \tfrac{1}{2} + i t_r + i v, f \right) \right|.
\]
For $\frac{13}{20} \le \sigma \le 1 - \varepsilon$, we denote by $S_1$ and $S_2$ the sets of points $s = \sigma + i t_r$ satisfying \eqref{eq4.51} and \eqref{eq4.52}, respectively, and we set $R_i := |S_i|$.

First, consider the contribution from $S_1$. By Lemma \ref{le4.13},
\begin{equation}\label{eq4.53}
R_1 \ll T^\varepsilon \left( Y^{2 - 2\sigma} V^{-4} + T V^{-2l(\sigma)} \right).
\end{equation}
Next, for the set $S_2$, applying Lemma \ref{le4.10} and Lemma \ref{le4.15} we obtain
\[
\begin{aligned}
R_2 &\ll Y^{\frac{1}{2} - \sigma} V^{-2} \sum_{t_r \in S_2} \left| L\left( \tfrac{1}{2} + i t_r', \mathrm{sym}^2 f \right) \right|^2 \\
&\ll Y^{\frac{1}{2} - \sigma} V^{-2} \left( R \log T + \log T \sum_{r \le R} \int_{-\log^2 T}^{\log^2 T} \left| L\left( \tfrac{1}{2} + i t_r + i v, f \right) \right|^2   dv \right) \\
&\ll Y^{\frac{1}{2} - \sigma} V^{-2} \left( T^{1+\varepsilon} + \log T \int_{1}^{2T + \log^2 T} \left| L\left( \tfrac{1}{2} + i t, f \right) \right|^2   dt \right) \\
&\ll Y^{\frac{1}{2} - \sigma} V^{-2} T^{\frac{3}{2} - \frac{3}{32} + \varepsilon}.
\end{aligned}
\]

Let $Y = V^{\frac{4}{3-2\sigma}} T^{\frac{3 - \frac{3}{16}}{3-2\sigma}}$. Then Lemma \ref{le4.14}, we get
\[
\begin{aligned}
R &\ll T^\varepsilon \left( Y^{2-2\sigma} V^{-4} + T V^{-2l(\sigma)} + Y^{\frac{1}{2}-\sigma} V^{-2} T^{\frac{3}{2} - \frac{3}{32}} \right) \\
  &\ll T^\varepsilon \left( V^{-\frac{4}{3-2\sigma}} T^{\frac{6-6\sigma - \frac{3}{8}(1-\sigma)}{3-2\sigma}} + T V^{-2l(\sigma)} \right) \\
  &\ll T^{1+\varepsilon} V^{-   \frac{2}{3} \cdot \frac{3-2\sigma + \frac{3}{8}(1-\sigma)}{(3-2\sigma)(1-\sigma)}}.
\end{aligned}
\]
This completes the proof of the desired estimate $R \ll T^{1+\varepsilon} V^{-m(\sigma)}$.
\end{proof}

Finally, it suffices to choose $\sigma_1=\frac{13}{20}$ in Lemma \ref{pro4.1} to obtain Theorem \ref{th1.3}.

\section{The Proof of Theorem \ref{th1.2}}

 We use Heath-Brown's method \cite{MR1159354} to prove Theorem \ref{th1.2}, i.e., to show that $F_f(t)$ satisfies Lemma \ref{le5.1},

\begin{lemma}\label{le5.1}

Let $F(t)$ be a real-valued function and $a_1(t), a_2(t), \ldots$ be continuous real-valued functions of period 1 , and suppose that there exist non-zero constants $\gamma_1, \gamma_2, \ldots$ which are linearly independent over $\mathbb{Q}$ such that
\[
\lim _{N \rightarrow \infty} \limsup _{T \rightarrow \infty} \frac{1}{T} \int_0^T \min \left\{1,\left|F(t)-\sum_{n \leqslant N} a_n\left(\gamma_n t\right)\right|\right\} d t=0.
\]
 Suppose moreover that
\[
\begin{aligned}
& \int_0^1 a_n(t) d t=0 \quad(n \in \mathbb{N}), \\
& \sum_{n=1}^{\infty} \int_0^1 a_n(t)^2 d t<\infty,
\end{aligned}
\]
and that there is a constant $\mu>1$ for which
\[
\max _{t \in[0,1]}\left|a_n(t)\right| \ll n^{1-\mu},
\]
and
\[
\lim _{n \rightarrow \infty} n^\mu \int_0^1 a_n(t)^2 d t=\infty.
\]
Then $F(t)$ has a distribution function $g(\alpha)$  in the sense that, for any interval I we have
\[
X^{-1} \operatorname{meas}\left\{t \in[1, X]: F(t) \in I\right\} \rightarrow \int_I g(\alpha) d \alpha,
\]
as $X \rightarrow \infty$. The function $g(\alpha)$ and its derivatives satisfy the growth condition
\[
\frac{d^k}{d \alpha^k} g(\alpha) \ll_{A, k}(1+|\alpha|)^{-A},
\]
for $k=0,1,2, \ldots$ and any constant A. Moreover, $g(\alpha)$ extends to an entire function on $\mathbb{C}$.
\end{lemma}
\begin{proof}
    See \cite[Theorem 5]{MR1159354} for details.
\end{proof}

Now we present the proof of Theorem \ref{th1.2}.
\begin{proof}
Let $\omega(x)$ be a non-negative $C^\infty$ function such that $\operatorname{supp}(\omega) \subset \left[\frac{1}{2}, \frac{17}{2}\right]$ and $\omega(x) \equiv 1$ on the interval $[1, 8]$. We will show that
$$
\limsup _{T \rightarrow \infty} \frac{1}{T} \int_0^{\infty}\left|\mathcal{F}(t)-\sum_{n \leqslant N} a_n\left(\gamma_n t\right)\right|^2 \omega\left((t / T)^3\right) d t \rightarrow 0,
$$
as $N$ tends to infinity, where $\mathcal{F}(t):=\mathcal{F}_f(t)=t^{-1}\Delta_f(t^3)$ and
$$
a_n(t):=a_{f,n}(t)=\frac{\varepsilon(n)}{\pi \sqrt{3}} \sum_{r=1}^{\infty} \frac{A_f\left(1,n r^3\right)}{\left(n r^3\right)^{2 / 3}} \cos \{6 \pi r t\}, \quad \gamma_n=\sqrt[3]{n}
$$
with $\varepsilon(n)=1$ if $n$ is cube-free, and $\varepsilon(n)=0$ otherwise. From \cite[Theorem 1.2]{MR3692033} and H\"{o}lder's inequality, we know that
\[
\begin{aligned}
    \sum_{r\leqslant x}A_f(1,nr^3)&\leqslant \left(\sum_{r\leqslant nx^3}A_f(1,r)^4\right)^{\frac{1}{4}}\left(\sum_{r\leqslant nx^3}\textbf{1}_{r/n \textbf{ is cube number} }^{\frac{4}{3}}\right)^\frac{3}{4}\\
    &\leqslant n^{\frac{1}{4}+\varepsilon}x^{\frac{3}{2}+\varepsilon}.
\end{aligned}
\]
From  above and summation by parts, it follows that $a_n(t)$ is well-defined.
\begin{remark}
    \cite[Theorem 1.2]{MR3692033}  requires that $f$ must be self-dual, so it can only be stated that $a_n(t)$ is well-defined for self-dual $f$. Moreover, since the Fourier coefficients of a Hecke–Maass cusp form $f$ lack good pointwise estimates, in the subsequent proof we will use the above expression and average estimates instead of pointwise estimates.
\end{remark}
From  integration by parts and Theorem \ref{th1.3}, it follows that
\begin{equation}\label{eq5.1}
    \int_0^{\infty} \mathcal{F}(t)^2 \omega\left((t / T)^3\right) d t=\frac{1}{6 \pi^2}\left\{\sum_{n=1}^{\infty} A_f(1,n)^2 n^{-4 / 3}\right\} J+O\left(T^{10/11+\varepsilon}\right),
\end{equation}
where
\[
J=\int_0^{\infty} \omega\left((t / T)^3\right) d t.
\]

We now estimate the integral
\[
\int_0^{\infty} \bigg| \sum_{n \leqslant N} a_n(\gamma_n t) \bigg|^2   \omega \left((t/T)^3\right) dt, \qquad (N \leqslant T^{1/6}).
\]
 It can be seen from a simple calculation that
\[
\int_0^Y \cos(6\pi y \sqrt[3]{n})   \cos(6\pi y \sqrt[3]{m})   dy = 
\begin{cases} 
Y/2 + O(1), & m = n, \\[10pt]
O \left( |\sqrt[3]{m} - \sqrt[3]{n}|^{-1} \right), & m \neq n.
\end{cases}
\]
Applying integration by parts yields, for $m=n$, 
\begin{equation}\label{eq5.2}
    \int_0^{\infty} \cos^2(6\pi t \sqrt[3]{n})   \omega \left((t/T)^3\right) dt = \frac{1}{2} J + O(1),
\end{equation}
and for $m\neq n$
\begin{equation}\label{eq5.3}
    \begin{aligned}
        \int_0^{\infty} \cos(6\pi t \sqrt[3]{n}) \cos(6\pi t \sqrt[3]{m})   \omega \left((t/T)^3\right) dt &\ll |\sqrt[3]{m} - \sqrt[3]{n}|^{-1}\\ \ll 
\begin{cases}
(mn)^{1/3}   |m-n|^{-1}, & m \ll n \ll m, \\
\min \left( m^{-1/3},   n^{-1/3} \right), & \text{otherwise}.
\end{cases}
    \end{aligned}
\end{equation}
From these bounds it follows directly that
\begin{equation}\label{eq5.5}
\int_0^{\infty} \bigg| \sum_{n \leqslant N} a_n(\gamma_n t) \bigg|^2   \omega \left((t/T)^3\right) dt = \sum_{n \leqslant N} S_n \;+\; O \left(N^{5/6+\varepsilon}\right),
\end{equation}
where
\[
S_n = J   \frac{\varepsilon(n)}{6\pi^2} \sum_{r=1}^{\infty} A_f(1, n r^3)^2   (n r^3)^{-4/3}.
\]

We now turn to the estimation of the cross terms
\[
\int_0^{\infty} \mathcal{F}(t) a_n(\gamma_n t)   \omega \left((t/T)^3\right) dt.
\]
Let $X=T^3$, then the integral with respect to $t$ in the above expression can be written as
\[
I_f(n, X) := \int_0^{\infty} \Delta_f(x) \cos\{6\pi \sqrt[3]{n x}\} \omega(x / X)  \frac{dx}{x}.
\]
Perron's formula implies that for any non-integer $x >0$, 
\begin{equation*}
    \Delta_f(x)=\lim_{T \to \infty}\frac{1}{2\pi i} \int_{(2)}L(s, f)\frac{x^s}{s} ds,
\end{equation*}
and the convergence is uniform for $x \in [\frac{X}{2}, \frac{17X}{2}]$. Consequently, by Lebesgue's Bounded Convergence Theorem,
\begin{equation}\label{eq5.6}
I_f(n, X) = \lim_{T \to \infty} \frac{1}{2\pi i} \int_{(2)} L(s, f)   K(s)   \frac{ds}{s},
\end{equation}
where
\[
K(s) := \int_0^{\infty} x^{s-1} \cos\{6\pi \sqrt[3]{n x}\}   \omega(x / X)   dx.
\]
Clearly $K(s)$ is an entire function. Repeated integration by parts ($k$ times) yields the estimate
\begin{equation}\label{eq5.7}
\begin{aligned}
K(\sigma + it)
\ll_{k,\sigma,\omega} (1+|t|)^{-k} X^\sigma (n X)^{k/3}.
\end{aligned}
\end{equation}
By the residue theorem, we have
\[
\begin{aligned}
I_f(n, X) &= L(0, f) K(0) + \frac{1}{2\pi i} \int_{-1-i\infty}^{-1+i\infty} L(s, f) K(s) \frac{ds}{s} \\
&= L(0, f) K(0) + \frac{1}{2\pi i} \int_{-1-i\infty}^{-1+i\infty} L(1-s, \tilde{f}) \chi_f(s) K(s) \frac{ds}{s},
\end{aligned}
\]
where $\chi(s)$ is the factor from the functional equation $L(s, f) = \chi_f(s) L(1-s, \tilde{f})$, given by
\[
\begin{aligned}
\chi_f(s) = & \; 2^{3s} \pi^{3(s-1)}   \Gamma(1-s)   \Gamma(2-s-3v_f)   \Gamma(-s+3v_f) \\
          & \times \sin \left(\frac{s}{2}\pi\right) \sin \left(\frac{s-1+3v_f}{2}\pi\right) \sin \left(\frac{s+1-3v_f}{2}\pi\right).
\end{aligned}
\]
We now expand $L(1-s,\tilde{f})$ into its Dirichlet series,
\[
L(1-s, \tilde{f}) = \sum_{m=1}^\infty A_f(1, m)   m^{ s-1}.
\]
Inserting this into the expression for $I(n,X)$. When $m> X^{1/2}$ the contribution is  
\[
\begin{aligned}
&\ll \left|A_f(1,m)\right| m^{-2} \Bigg\{ \int_0^{(nX)^{1/3}}
\frac{(1+t)^{7/2}} {X} dt + \int_{(nX)^{1/3}}^\infty \frac{n^{5/3} X^{2/3}}{t^{3/2}} dt \Bigg\} \\
&\ll \left|A_f(1,m)\right| m^{-2}   n^{3/2} X^{1/2},
\end{aligned}
\]
by using \eqref{eq5.7} with $k=0$ and $k=5$.
Summing over $m > X^{1/2}$ and using the bound $\sum_{m\leqslant X}A(1,m) \ll X$, the total contribution from these terms is 
\[
O \left( n^{3/2} X^{1/2+\varepsilon} \sum_{m > X^{1/2}} \frac{\left|A_f(1,m)\right|}{m^2} \right) = O \left( n^{3/2} X^{\varepsilon} \right) = O \left( X^{1/4+\varepsilon} \right),
\]
since $n \le N \ll T^{1/6} = X^{1/18}$.

For the remaining terms with $m \le X^{1/2}$, we shift the line of integration to $\operatorname{Re}(s) = 1/6 + \varepsilon$. Then
\begin{equation}\label{eq5.8}
\begin{aligned}
&\frac{1}{2\pi i} \int_{-1-i\infty}^{-1+i\infty} m^{s-1} \chi_f(s) K(s)   \frac{ds}{s} \\
&= \int_0^{\infty} \cos\{6\pi\sqrt[3]{nx}\}   \omega(x/X) \Bigg\{ \frac{1}{2\pi i} \int_{1/6+\varepsilon-i\infty}^{1/6+\varepsilon+i\infty} (mx)^{s-1} \chi_f(s)   \frac{ds}{s} \Bigg\} dx.
\end{aligned}
\end{equation}
By Stirling's formula,
\[
\begin{aligned}
\chi_f(s) s^{-1} &= \frac{1}{2\pi}   3^{3s-1/2}   \chi(3s) \left\{ 1 + O \left(|s|^{-1}\right) \right\} \\
&= \frac{1}{2\pi}   3^{3s-1/2}   \chi(3s) + O \left( |s|^{-1-\varepsilon} \right),
\end{aligned}
\]
where
\[
\chi(s) = 2^s \pi^{ s-1} \sin \left( \frac{\pi s}{2} \right) \Gamma(1-s).
\]
Moreover, for $1/2 < \sigma < 1$ and any $x>0$,
\[
\frac{1}{2\pi i} \int_{\sigma-i\infty}^{\sigma+i\infty} x^s   \chi(s)   ds = 2x \cos(2\pi x).
\]
Therefore,
\[
\begin{aligned}
&\frac{1}{2\pi i} \int_{1/6+\varepsilon-i\infty}^{1/6+\varepsilon+i\infty} (mx)^s   \chi_f(s)   \frac{ds}{s} \\
&= \frac{1}{2\pi i} \int_{1/2+3\varepsilon-i\infty}^{1/2+3\varepsilon+i\infty} \frac{1}{6\pi\sqrt{3}} \left( 3(mx)^{1/3} \right)^{ s} \chi(s)   ds \;+\; O \left( (mx)^{1/6+\varepsilon} \right) \\
&= \frac{(mx)^{1/3}}{\pi\sqrt{3}} \cos \left\{ 6\pi (mx)^{1/3} \right\} \;+\; O \left( (mx)^{1/6+\varepsilon} \right).
\end{aligned}
\]
The error term $O \left((mx)^{1/6+\varepsilon}\right)$ derived above contributes
\[
O \left( m^{\varepsilon-5/6} X^{\varepsilon+1/6} \right)
\]
to the integral in \eqref{eq5.8}. Summing this over all $m \le X^{1/2}$ yields a total contribution of $O \left(X^{1/4+2\varepsilon}\right)$ to $I_f(n, X)$.

We now treat the main term. Observing that $X = T^3$ and making the change of variable $x = t^3$, we have
\[
\begin{aligned}
&\int_0^{\infty} \cos \left\{6\pi\sqrt[3]{n x}\right\} \cos \left\{6\pi\sqrt[3]{m x}\right\} x^{-2/3} \omega(x/X)   dx \\
&= 3 \int_0^{\infty} \cos \left\{6\pi t\sqrt[3]{n}\right\} \cos \left\{6\pi t\sqrt[3]{m}\right\} \omega \left((t/T)^3\right) dt.
\end{aligned}
\]
Combining this identity with the estimates \eqref{eq5.2} and \eqref{eq5.3} for the last integral, we deduce that the contribution of the main term $\frac{(mx)^{1/3}}{\pi\sqrt{3}} \cos\left\{6\pi (mx)^{1/3}\right\}$ to \eqref{eq5.8} is
\[
\begin{cases}
\dfrac{\sqrt{3}}{2\pi}   n^{-2/3} J + O\left(n^{-2/3}\right), & m = n, \\[8pt]
\ll |m - n|^{-1}, & m \ne n \text{ and } m \ll n \ll m, \\[8pt]
\ll m^{-2/3} \min\left(m^{-1/3}, n^{-1/3}\right), & \text{otherwise}.
\end{cases}
\]
Summation over $m \le X^{1/2}$ shows that the contribution of error terms to $I(n, X)$ is $O\left(X^{\varepsilon}\right)$. Therefore,
\begin{equation}\label{eq:I-estimate}
I(n, X) = \frac{\sqrt{3}}{2\pi}   A_f(1, n)   n^{-2/3} J + O\left(X^{1/4+\varepsilon}\right).
\end{equation}

Reverting to the original variable $t$ via $x = t^3$, we obtain from \eqref{eq:I-estimate} the uniform estimate
\[
\int_0^{\infty} \mathcal{F}(t) \cos\left\{6\pi t\sqrt[3]{n}\right\} \omega \left((t/T)^3\right) dt
= \frac{1}{2\pi\sqrt{3}}   A_f(1, n)   n^{-2/3} J + O \left(T^{3/4+\varepsilon}\right),
\]
for all $n \le N \le T^{1/6}$. It follows that
\[
\begin{aligned}
&\int_0^{\infty} \mathcal{F}(t)   \frac{\varepsilon(n)}{\pi\sqrt{3}} \sideset{}{'}\sum_r \frac{A_f\left(1,n r^3\right)}{\left(n r^3\right)^{2/3}} \cos\left\{6\pi t r\sqrt[3]{n}\right\} \omega\left((t/T)^3\right) dt \\
&= \frac{\varepsilon(n)}{6\pi^2}   J \sideset{}{'}\sum_r \frac{A_f\left(1,n r^3\right)^2}{\left(n r^3\right)^{4/3}}
+ O\left(n^{-2/3+\varepsilon} T^{3/4+\varepsilon}\right) \\
&= S_n + O \left(n^{-1/3+\varepsilon}T^{-1/21+\varepsilon} J\right) + O\left(n^{-2/3+\varepsilon} T^{3/4+\varepsilon}\right) \\
&= S_n + O\left(n^{-1/3+\varepsilon} T^{20/21+\varepsilon}\right),
\end{aligned}
\]
where the sum $\sum'$ is restricted to $r$ satisfying $n r^3 \le T^{1/6}$.
Furthermore,
\[
a_n(t) - \frac{\varepsilon(n)}{\pi\sqrt{3}} \sideset{}{'}\sum_r \frac{A_f\left(1,n r^3\right)}{\left(n r^3\right)^{2/3}} \cos \left\{6\pi r t\right\}
\ll n^{-1/4+\varepsilon} T^{-1/36+\varepsilon}.
\]
Therefore,
\[
\begin{aligned}
&\int_0^{\infty} \mathcal{F}(t) a_n(\gamma_n t)   \omega \left((t/T)^3\right) dt \\
&= S_n + O \left(n^{-1/3+\varepsilon} T^{25/26+\varepsilon}\right) \\
&\quad + O \left( n^{-1/4+\varepsilon} T^{-1/36+\varepsilon} \int_0^{\infty} |\mathcal{F}(t)|   \omega \left((t/T)^3\right) dt \right).
\end{aligned}
\]

An application of Cauchy--Schwarz inequality to \eqref{eq5.1} yields
\[
\int_0^{\infty} |\mathcal{F}(t)|   \omega \left((t/T)^3\right) dt \ll T,
\]
whence the second error term above is $O \left(n^{-1/4+\varepsilon} T^{35/36+\varepsilon}\right)$. Summing over $n \le N$, we conclude that
\[
\int_0^{\infty} \mathcal{F}(t) \sum_{n \le N} a_n(\gamma_n t)   \omega \left((t/T)^3\right) dt
= \sum_{n \le N} S_n + O \left(T^{71/72+\varepsilon}\right),
\]
for $N \le T^{1/54}$.

Finally, combining this result with \eqref{eq5.1} and \eqref{eq5.5}, we obtain
\begin{equation}\label{eq510}
\begin{aligned}
\int_0^{\infty} &\left| \mathcal{F}(t) - \sum_{n \le N} a_n(\gamma_n t) \right|^2 \omega \left((t/T)^3\right) dt \\
&= \sum_{n>N} S_n + O \left(T^{71/72+\varepsilon}\right) \\
&\ll T N^{\varepsilon-1/3} + T^{71/72+\varepsilon} \\
&\ll T N^{-1/4},
\end{aligned}
\end{equation}
for $N \le T^{1/54}$. Applying Lemma \ref{le5.1} with $\mu = 7/5$ completes the proof of Theorem \ref{th1.2}.
\end{proof}

\section{The Proof of Theorem \ref{th1.4}}
We first derive upper bounds for every integral moment of the truncated random processes \( F_{f,\mathbb{X}} \).
\begin{lemma}\label{le6.1}
    Let $f$ be a fixed self-dual Hecke--Maass cusp form for $\operatorname{GL}(3)$. Let $\beta_0$ be a real constant. Let $\left\{a_n\right\}_n$ be a sequence of real numbers such that $\left|a_n\right| \leqslant 1$. There exist positive constants $c_{f,1}, c_{f,2}$ such that for any positive integers $1 \leqslant N<M, L \geqslant 2$ and $k \geqslant 1$ we have
    \begin{equation}\label{eq6.1}
    \begin{aligned}
        & \mathbb{E}\left(\left|\sum_{N \leqslant n<M} \frac{\varepsilon(n)}{n^{2 / 3}} \sum_{q \leqslant L} a_{n q^3} \frac{A_f\left(1,n q^3\right)}{q^{2}} \cos \left(6 \pi q \mathbb{X}_n\right)\right|^k\right) \\
        & \leqslant \min \left\{\left(c_{f,1} k^{2 / 3}\left(\log 2k\right)^{2/3}\right)^k,\left(c_{f,2} kN^{-1/3}\right)^{k / 2}\right\}.
    \end{aligned}
    \end{equation}
\end{lemma} 
\begin{proof}
 We now note that
\[
\begin{aligned}
& \mathbb{E}\left(\left|\sum_{N \leqslant n<M} \frac{\varepsilon(n)}{n^{2 / 3}} \sum_{q \leqslant L} a_{n q^3} \frac{A_f\left(1,n q^3\right)}{q^{2}} \cos \left(6 \pi q \mathbb{X}_n\right)\right|^{2 k}\right) \\
& =\mathbb{E}\left(\left|\operatorname{Re}\left( \sum_{N \leqslant n<M} \frac{\varepsilon(n)}{n^{2 / 3}} \sum_{q \leqslant L} a_{n q^3} \frac{A_f\left(1,n q^3\right)}{q^{ 2}} e^{6 \pi i q \mathbb{X}_n}\right)\right|^{2k}\right) \\
& \leqslant \mathbb{E}\left(\left|\sum_{N \leqslant n<M} \frac{\varepsilon(n)}{n^{2 / 3}} \sum_{q \leqslant L} a_{n q^3} \frac{A_f\left(1,n q^3\right)}{q^{ 2}} e^{6 \pi i q \mathbb{X}_n}\right|^{2 k}\right) .
\end{aligned}
\]
Expanding this moment we find that it equals
\[
\begin{aligned}
& \sum_{N \leqslant n_1, n_2, \ldots, n_{2 k}<M} \frac{\varepsilon\left(n_1\right) \cdots 
\varepsilon\left(n_{2 k}\right)}{\left(n_1 \cdots n_{2 k}\right)^{2 / 3}} \sum_{q_1, q_2, \ldots, q_{2 k} \leqslant L} \prod_{j=1}^{2 k} a_{n_j q_j^3} \frac{A_f\left(1,n_j q_j^3\right)}{q_j^{2}} \\
& \quad \times \mathbb{E}\left(e^{2 \pi i\left(\left(q_1 \mathbb{X}_{n_1}+\cdots+q_{k} \mathbb{X}_{n_{k}}\right)-\left(q_{k+1} \mathbb{X}_{n_{k+1}}+\cdots+q_{1k} \mathbb{X}_{n_{2k}}\right)\right)}\right).
\end{aligned}
\]
Since the random variables $\{\mathbb{X}_n\}$ are independent, the inner expectation vanishes unless
\[
\{n_1,\dots,n_k\}=\{n_{k+1},\dots,n_{2k}\}
\] 
and
\[
q_1\mathbb{X}_{n_1}+\cdots+q_k\mathbb{X}_{n_k}=q_{k+1}\mathbb{X}_{n_{k+1}}+\cdots+q_{2k}\mathbb{X}_{n_{2k}},
\]
in which case it equals $1$. Moreover, under this condition we must have $q_i = q_j$ whenever $n_i = n_j$ for some $1 \le i \le k$ and $k+1 \le j \le 2k$. Therefore, noting that 
\[
\sum_{n\geqslant N} \sum_{r}\frac{\varepsilon (n)^2 A_f(1,nr^3)^2}{\left(nr^3\right)^{4/3}}\leqslant \sum_{n>N}\frac{A_f(1,n)^2}{n^{4/3}},
\]
we have
\[
\begin{aligned}
&\mathbb{E}\left(\left|\sum_{N \leqslant n<M} \frac{\varepsilon(n)}{n^{2/ 3}} \sum_{q \leqslant L} a_{n q^3} \frac{A_f\left(1,n q^3\right)}{q^{ 2}} e^{ 6\pi i q \mathbb{X}_n}\right|^{2 k}\right)\\
& \leqslant k!\left(\sum_{N \leqslant n<M} \frac{\varepsilon(n)}{n^{4/3}} \sum_{q \leqslant L} \frac{A_f\left(1,n q^3\right)^2}{q^4}\right)^k   \\
&  \ll\left(k N^{-1/3}\right)^k.
\end{aligned}
\]
Define $J := k\log 2k$. To complete the proof of \eqref{eq6.1}, we may assume $N < J < M$. Indeed, if $N \ge J$ the result follows from the previous bound, whereas the case $M \le J$ is trivial. Since 
\[
\sum_{r\leqslant x}A_f(1,nr^3) \leqslant n^{\frac{1}{4}+\varepsilon}x^{\frac{3}{2}+\varepsilon},
\]
then we have
\[
\left|\sum_{N \leqslant n<M} \frac{\varepsilon(n)}{n^{2 / 3}} \sum_{q \leqslant L} a_{n q^3} \frac{A_f\left(1,n q^3\right)}{q^{2}} e^{6 \pi i q \mathbb{X}_n}\right|  \ll k^{2 / 3}(\log 2k)^{2/3}.
\]
We now apply the elementary inequality
\[
|a+b|^{2\ell} \leqslant 2^{2\ell}\bigl( |a|^{2\ell} + |b|^{2\ell} \bigr),
\]
which holds for all real numbers $a, b$ and any positive integer $\ell$.  Hence, when $N<J<M$, we have
\[
\begin{aligned}
& \mathbb{E}\left(\left|\sum_{N \leqslant n<M} \frac{\varepsilon(n)}{n^{2 / 3}} \sum_{q \leqslant L} a_{n q^3} \frac{A_f\left(1,n q^3\right)}{q^{2}} e^{6 \pi i q \mathbb{X}_n}\right|^{2 k}\right) \\
& \leqslant 2^{2 k} \mathbb{E}\left(\left|\sum_{N \leqslant n<J} \frac{\varepsilon(n)}{n^{2/3}} \sum_{q \leqslant L} a_{n q^3} \frac{A_f\left(1,n q^3\right)}{q^{2}} e^{6\pi i q \mathbb{X}_n}\right|^{2 k}\right) \\
& \quad+2^{2 k} \mathbb{E}\left(\left|\sum_{J \leqslant n<M} \frac{\varepsilon(n)}{n^{2/3}} \sum_{q \leqslant L} a_{n q^3} \frac{A_f\left(1,n q^3\right)}{q^{2}} e^{6 \pi i q \mathbb{X}_n}\right|^{2 k}\right) \\
& \leqslant\left(c^\prime_{f,1} k^{2/3}\left(\log 2k\right)^{2/3}\right)^{2 k}+\left(c^\prime_{f,2} kJ^{-1/3}\right)^k.
\end{aligned}
\]
for some positive constants $c^\prime_{f,1}, c^\prime_{f,2}$. Finally, by applying Cauchy--Schwarz inequality, $\mathbb{E}(|Z|^k) \leqslant \mathbb{E}(|Z|^{2k})^{1/2}$ (for any random variable $Z$), we complete the proof of the Lemma.
\end{proof}

Next, we establish the following results. They show that, in a certain uniform range, the moments of the truncations of $F_f(t)$ are very close to those of the corresponding probabilistic random model.
\begin{lemma}\label{le6.2}
Let $m \geqslant 1$ and $n_1, \ldots, n_m \leqslant M$ be positive integers. Let $\varepsilon_j= \pm 1$ be such that $\sum_{j=1}^m \varepsilon_j n_j^{1/3} \neq 0$. Then we have
\[
\left|\sum_{j=1}^m \varepsilon_j \sqrt{n_j}\right| \geqslant \frac{1}{(m M^{1/3})^{3^{m-1}-1}} .
\]
\end{lemma}
\begin{proof}
   The proof is similar to \cite[Lemma 3.5]{MR2039790}.
\end{proof}

\begin{lemma}\label{le6.3}
     Let $a_m$ be real numbers such that $a_m \ll 1$. Let $\alpha_0 \neq 0$ and $\beta_0$ be fixed real numbers. Let $T$ be large, $0 \leqslant h \leqslant(\log \log T) / 4$ be an integer and $M \leqslant T^{1 /\left(3^h+9 h\right)} / h^3$ be a real number. Write $m=n r^3$ where $n$ is cube-free and let $\left\{\mathbb{X}_n\right\}_{n \text { cube-free  }}$be a sequence of independent random variables uniformly distributed on $[0,1]$. Then we have
     \[
     \begin{aligned}
     &\frac{1}{T} \int_T^{2 T}\left(\sum_{m \leqslant M} a_m \cos \left(6 \pi \alpha_0 \sqrt[3]{m t}\right)\right)^h d t\\
     &=  \mathbb{E}\left(\left(\sum_{n r^2 \leqslant M} a_{n r^2} \cos \left(6 \pi r \mathbb{X}\right)\right)^h\right) +O\left(T^{- 2/ 9}\right).
     \end{aligned}
     \]
\end{lemma}
\begin{proof}
Let $e(z)=\exp(2\pi iz)$. It suffices to prove that, for all integers $k, \ell \geqslant 0$ and $k+\ell \leqslant h$,
\[
\begin{aligned}
& \frac{1}{T} \int_T^{2 T}\left(\sum_{q \leqslant M} a_q e\left(3\alpha_0 \sqrt[3]{q t}\right)\right)^k{\overline{\left(\sum_{m \leqslant M} a_m e\left(3\alpha_0 \sqrt[3]{m t}\right)\right)}}^{\ell} d t \\
& =\mathbb{E}\left(\left(\sum_{n_1 r_1^3 \leqslant M} a_{n_1 r_1^3} e^{6 \pi i r_1 \mathbb{X}\left(n_1\right)}\right)^k \overline{\left(\sum_{n_2 r_2^3 \leqslant M} a_{n_2 r_2^3} e^{6 \pi i r_2 \mathbb{X}\left(n_2\right)}\right)}^{\ell}\right)\\
&\quad +O\left(T^{-2 / 9}\right).
\end{aligned}
\]
 Expanding the moment we obtain
\[
\begin{aligned}
& \left.\frac{1}{T} \int_T^{2 T}\left(\sum_{q \leqslant M} a_q e\left(3\alpha_0 \sqrt[3]{q t}\right)\right)^k \overline{\left(\sum_{m \leqslant M} a_m e\left(3\alpha_0 \sqrt[3]{m t}\right)\right.}\right)^{\ell} d t \\
& =\sum_{\substack{q_1, \ldots, q_k \leqslant M \\
m_1, \ldots, m_{\ell} \leqslant M}} \prod_{i=1}^k a_{q_i} \prod_{j=1}^{\ell} \overline{a_{m_j}} \frac{1}{T} \int_T^{2 T} e\left(3\alpha_0\left(\sum_{i=1}^k \sqrt[3]{q_i}-\sum_{j=1}^{\ell} \sqrt[3]{m_j}\right) \sqrt{t}\right) d t.
\end{aligned}
\]
Let $q_i = d_i e_i^3$ and $m_j = f_j g_j^3$, where $d_i$ and $f_j$ is cube-free. Since the set $\{\sqrt{n}\}_{n\text{ cube-free}}$ is linearly independent over $\mathbb{Q}$ (see \cite[Theorem 2]{MR2327}) and the random variables $\mathbb{X}(n)$ are independent, we conclude that the contribution from the diagonal terms those satisfying
\[
\sum_{i=1}^k \sqrt[3]{q_i} - \sum_{j=1}^{\ell} \sqrt[3]{m_j} = 0,
\]
equals
\[
\mathbb{E}\left(\left(\sum_{n_1 r_1^3 \leqslant M} a_{n_1 r_1^3} e^{6 \pi i r_1 \mathbb{X}\left(n_1\right)}\right)^k \overline{\left(\sum_{n_2 r_2^2 \leqslant M} a_{n_2 r_2^3} e^{6 \pi i r_2 \mathbb{X}\left(n_2\right)}\right)^{\ell}}\right) .
\]
To finish the proof of the lemma, it remains to bound the contribution of the off-diagonal terms those with
\[
\sum_{i=1}^k \sqrt[3]{q_i} - \sum_{j=1}^{\ell} \sqrt[3]{m_j} \neq 0.
\]
For each such term where all $q_1, \ldots, q_k, m_1, \ldots, m_{\ell} \leqslant M$, Lemma \ref{le6.3} gives
\[
\left|\sum_{i=1}^k \sqrt[3]{q_i}-\sum_{j=1}^{\ell} \sqrt[3]{m_j}\right| \geqslant \frac{1}{((k+\ell) \sqrt[3]{M})^{3^{k+\ell-1}-1}} .
\]
A simple integration by parts shows that for any nonzero real $\eta$,
\[
\begin{aligned}
\int_T^{2 T} e(\eta \sqrt[3]{t}) d t\ll \frac{T^\frac{2}{3}}{|\eta|}.
\end{aligned}
\]
Thus we deduce that the contribution of the off-diagonal terms  is
\[
\ll \frac{1}{\sqrt[3]{T}}\left(c M\right)^{k+\ell}((k+\ell) \sqrt[3]{M})^{3^{k+\ell-1}-1} \ll T^{-2 / 9},
\]
for some positive constant $c$.
\end{proof}

Finally, we establish the following results. They demonstrate that, uniformly in a certain range of parameters, the Laplace transform of $F_f(t)$ (over an appropriate set of full measure) closely approximates that of the probabilistic random model $F_{f,\mathbb{X}}$.
\begin{lemma}\label{le6.4}
    Let $T$ be large and let $f$ be a fixed self-dual Hecke--Maass cusp form for $\operatorname{GL}(3)$. For a positive integer $N$, we define
\[
F_{f,N}(t):=\frac{1}{\pi \sqrt{3}} \sum_{n \leqslant N} \frac{\varepsilon(n)}{n^{2/3}} \sum_{r \leqslant N} \frac{A_f\left(1,n r^3\right)}{r^{ 2}} \cos \left(6 \pi r \sqrt[3]{n t}\right).
\]
If $N \leqslant T^{1 /162}$ then for any fixed $\varepsilon>0$ we have
\[
\frac{1}{T} \int_T^{2 T}\left|F_f(t)-F_{f,N}(t)\right| d t \ll \frac{1}{N^{1 / 8}}.
\]
\end{lemma}

\begin{proof}
Define
\[
G_{f,N}(t) := \frac{1}{\pi \sqrt{3}} \sum_{n \leqslant N} \frac{\varepsilon(n)}{n^{2/3}} \sum_{r=1}^{\infty} \frac{A_f(1, n r^3)}{r^{2}} \cos\bigl(6\pi r \sqrt[3]{n t}\bigr).
\]
Then, for $N \le U^{1/54}$, it follows from \eqref{eq510} that
\begin{equation}\label{eq6.2}
\int_{U}^{2U} \bigl| F_f(u^{3}) - G_{f,N}(u^{3}) \bigr|^{2}   du \ll\frac{U}{N^{1/4}}.
\end{equation}
Making the change of variable $t = u^{3}$ and applying the Cauchy--Schwarz inequality, we obtain
\begin{equation}\label{eq6.3}
\begin{aligned}
&\int_T^{2 T}\left|F_f(t)-G_{f,N}(t)\right| d t \\
& \leq 3\left(\int_{\sqrt[3]{T}}^{\sqrt[3]{2 T}}\left|F_f\left(u^3\right)-G_{f,N}\left(u^3\right)\right|^2 d u\right)^{1 / 2}\left(\int_{\sqrt[3]{T}}^{\sqrt[3]{2 T}} u^4 d u\right)^{1 / 2} \\
& \ll\frac{T}{N^{1 / 8}}.
\end{aligned}
\end{equation}
Finally, noting that
\[
 \sum_{r\leqslant x}A_f(1,nr^3) \leqslant n^{\frac{1}{4}+\varepsilon}x^{\frac{3}{2}+\varepsilon},
\]
 we derive
\[
\begin{aligned}
\sup _{t \in[T, 2 T]}\left|G_{f,N}(t)-F_{f,N}(t)\right| & \ll \sum_{n \leq N} \frac{1}{n^{2 / 3}} \sum_{r>N} \frac{A_f\left(1,n r^3\right)}{r^{ 2}} \\
&  \ll \frac{1}{N^{1 / 5}} .
\end{aligned}
\]
Combining this bound with \eqref{eq6.3} completes the proof.
\end{proof}

\begin{lemma}\label{le6.5}
    Let $f$ be a fixed self-dual Hecke--Maass cusp form for $\operatorname{GL}(3)$. There exists  a set $\mathcal{A} \in[T, 2 T]$ verifying $\operatorname{meas}([T, 2 T] \backslash \mathcal{A}) \ll T(\log T)^{-20}$ and a positive constant $c_{f,0}$, such that for all complex numbers $\lambda$ with
    \[
    |\lambda| \leqslant c_{f,0}\frac{\left(\log\log T\right)^{1/3}}{\left(\log\log\log T\right)^{5/3}},
    \] 
    we have
    \[
    \frac{1}{T} \int_{\mathcal{A}} \exp (\lambda F_f(t)) d t=\mathbb{E}\left(e^{\lambda F_{f,\mathbb{X}}}\right)+O\left(\exp \left(-\frac{\log \log  T}{\log\log \log  T}\right)\right).
    \]
\end{lemma}
\begin{proof}
Let $N: = \exp(\sqrt{\log T})$. Denote by $\mathcal{A}_1$ the set of all $t \in [T, 2T]$ for which
\begin{equation}\label{eq6.4}
\bigl| F_f(t) - F_{f,N}(t) \bigr| \le N^{-1/20}.
\end{equation}
From Lemma \ref{le6.4} we obtain
\begin{equation}\label{eq6.5}
    \operatorname{meas}\left([T, 2 T] \backslash \mathcal{A}_1\right) \leq N^{1 / 20} \int_T^{2 T}\left|F_f(t)-F_{f,N}(t)\right| d t \ll \frac{T}{N^{1 / 20}}.
\end{equation}
Let $K := \bigl\lfloor (\log\log T)/8 \bigr\rfloor$ and define $\mathcal{A}_2$ as the set of $t \in [T, 2T]$ satisfying
\begin{equation}\label{eq6.6}
\bigl| F_{f,N}(t) \bigr| \le V,
\end{equation}
where $V = c_{f,3} K^{2/3}\left(\log 2K\right)^{2/3}$ for a suitably large constant $c_{f,3}$. We also set
\[
F_{f,N, \mathbb{X}}:=\frac{1}{\pi \sqrt{3}} \sum_{n \leqslant N} \frac{\varepsilon(n)}{n^{2 / 3}} \sum_{r \leqslant N} \frac{A_f\left(1,n r^3\right)}{r^{ 2}} \cos \left(6\pi r \mathbb{X}_n\right),
\]
where $\left\{\mathbb{X}_n\right\}_{n \text { cube-free }}$ is a sequence of independent random variables, uniformly distributed on $[0,1]$. Then it follows from Lemma \ref{le6.3} that for any integer $0 \leqslant k \leqslant 2 K$ we have
\begin{equation}\label{eq6.7}
    \frac{1}{T} \int_T^{2 T} F_{f,N}(t)^k d t=\mathbb{E}\left(F_{f,N, \mathbb{X}}^k\right)+O\left(T^{-2 / 9}\right).
\end{equation}
Combining this and  Lemma \ref{le6.1}, we have
\begin{equation}\label{eq6.8}
    \begin{aligned}
\operatorname{meas}\left([T, 2 T] \backslash \mathcal{A}_2\right) & \leqslant V^{-2 K} \int_T^{2 T}\left|F_{f,N}(t)\right|^{2 K} d t \\
& \ll V^{-2 K} T \cdot \mathbb{E}\left(\left|F_{f,N, \mathrm{X}}\right|^{2 K}\right)+V^{-2 K} T^{7 / 9} \\
& \ll T\left(\frac{2 c_{f,1} K^{2 /3}\left(\log 2K\right)^{2/3}}{V}\right)^{2 K} \ll \frac{T}{(\log T)^{20}}.
\end{aligned}
\end{equation}
with a sufficiently large $c_{f,3}$. Define $\mathcal{A} := \mathcal{A}_1 \cap \mathcal{A}_2$ and  $\mathcal{A}^c = [T, 2T] \setminus \mathcal{A}$. Combining \eqref{eq6.5} and \eqref{eq6.8}, we obtain
\begin{equation}\label{eq6.9}
    \operatorname{meas}\left(\mathcal{A}^c\right) \ll \frac{T}{(\log T)^{20}} .
\end{equation}
Now by \eqref{eq6.4} and 
\[
\frac{1}{T} \int_{\mathcal{A}} \exp \left(\operatorname{Re}(\lambda) F_{f,N}(t)\right) d t \leqslant \exp (\operatorname{Re}(\lambda) V) \ll N^{1 / 50},
\]
we get
\begin{equation}\label{eq6.10}
    \begin{aligned}
\frac{1}{T} \int_{\mathcal{A}} \exp (\lambda F_f(t)) d t & =\frac{1}{T} \int_{\mathcal{A}} \exp \left(\lambda F_{f,N}(t)+O\left(|\lambda| N^{-1 / 20}\right)\right) d t \\
& =\frac{1}{T} \int_{\mathcal{A}} \exp \left(\lambda F_{f,N}(t)\right) d t+O\left(N^{-1 / 30}\right).
\end{aligned}
\end{equation}
We now estimate the main term on the right‑hand side of above. Set $L := \bigl\lfloor \frac{\log \log T}{\log \log \log T}  \bigr\rfloor$. By Stirling’s formula and the assumption on $\mathcal{A}$, we have
\begin{equation}\label{eq6.11}
    \frac{1}{T} \int_{\mathcal{A}} \exp \left(\lambda F_{f,N}(t)\right) d t=\sum_{k=0}^{2 L} \frac{\lambda^k}{k!} \frac{1}{T} \int_{\mathcal{A}} F_{f,N}(t)^k d t+E_1.
\end{equation}
where
\[
E_1 \ll \sum_{k>2 L} \frac{(|\lambda| V)^k}{k!} \ll \sum_{k>2 L}\left(\frac{3|\lambda| V}{k}\right)^k \ll \sum_{k>2 L}\left(\frac{3|\lambda| V}{2 L}\right)^k \ll e^{-L},
\]
with $c_{f,0}$ is sufficiently small. By the Cauchy--Schwarz inequality, Lemma \ref{le6.1}, \eqref{eq6.7} and \eqref{eq6.9} gives, for every $k \leqslant 2L$,
\[
\begin{aligned}
\left|\int_{\mathcal{A}^c} F_{f,N}(t)^k d t\right| & \leqslant \operatorname{meas}\left(\mathcal{A}^c\right)^{1 / 2}\left(\int_T^{2 T}\left|F_{f,N}(t)\right|^{2 k} d t\right)^{1 / 2} \\
& \ll \frac{\sqrt{T}}{(\log T)^{10}}\left(T \cdot \mathbb{E}\left(\left|F_{f,N, \mathrm{X}}\right|^{2 k}\right)+T^{7 / 9}\right)^{1 / 2} \\
& \ll \frac{T}{(\log T)^{10}}\left(2 c_{f,1} k^{2 / 3}\left(\log 2k\right)^{2/3}\right)^k \ll \frac{T}{(\log T)^9} .
\end{aligned}
\]
Inserting this estimate in \eqref{eq6.11} and using  \eqref{eq6.7} we get
\begin{equation}\label{eq6.12}
    \frac{1}{T} \int_{\mathcal{A}} \exp \left(\lambda F_{f,N}(t)\right) d t=\sum_{k=0}^{2 L} \frac{\lambda^k}{k!} \mathbb{E}\left(F_{f,N, \mathbb{X}}^k\right)+E_2,
\end{equation}
where
\[
E_2 \ll e^{-L}+\frac{1}{(\log T)^9} \sum_{k=0}^{2 L} \frac{|\lambda|^k}{k!} \ll e^{-L}+\frac{e^{|\lambda|}}{(\log T)^9} \ll e^{-L} .
\]
Combining above with \eqref{eq6.10} we derive
\begin{equation}\label{eq6.13}
    \frac{1}{T} \int_{\mathcal{A}} \exp (\lambda F_f(t)) d t=\sum_{k=0}^{2 L} \frac{\lambda^k}{k!} \mathbb{E}\left(F_{f,N, \mathbb{X}}^k\right)+O\left(e^{-L}\right).
\end{equation}
It remains to prove that the main term on the right‑hand side of \eqref{eq6.13} is asymptotically close to $\mathbb{E}\bigl(e^{\lambda F_{\mathbb{X}}}\bigr)$. For this purpose, define the event $\mathcal{B}$ by
\[
\left|F_{f,\mathrm{X}}-F_{f,N, \mathrm{X}}\right| \leqslant N^{-1 / 20} .
\]
Then, it follows from Lemma \ref{le6.1} that
\begin{equation}\label{eq6.14}
    \begin{aligned}
\mathbb{P}\left(\mathcal{B}^c\right) & \leqslant N^{1 / 10} \mathbb{E}\left(\left|F_{f,\mathbb{X}}-F_{f,N, \mathbb{X}}\right|^2\right) \\
& \ll N^{1 / 10} \mathbb{E}\left(\left|\sum_{n>N} \frac{\varepsilon(n)}{n^{2/3}} \sum_{r=1}^{\infty} \frac{A_f\left(1,n r^3\right)}{r^{2}} \cos \left(6 \pi r \mathbb{X}_n\right)\right|^2\right)+N^{-1 /5} \\
& \ll N^{-1 / 5},
\end{aligned}
\end{equation}
since, noting that
\[
\sum_{n\leqslant N} \sum_{r> N}\frac{\varepsilon(n)^2A_f(1,nr^3)^2}{\left(nr^3\right)^{4/3}}\leqslant \sum_{n>N}\frac{A_f(1,n)^2}{n^{4/3}}, 
\]
we have
\[
 \mathbb{E}\left(\left|\sum_{n\leqslant N} \frac{\varepsilon(n)}{n^{2/3}} \sum_{r>N}^{\infty} \frac{A_f\left(1,n r^3\right)}{r^{2}} \cos \left(6 \pi r \mathbb{X}_n\right)\right|^2\right)
 \ll N^{-1 /3} .
\]

Let $\mathbf{1}_{\mathcal{B}}$ denote the indicator function of an event $\mathcal{B}$, and let $k \leqslant 2L$ be a positive integer. We then note that
\[
\begin{aligned}
\mathbb{E}\left(F_{f,\mathbb{X}}^k\right) & =\mathbb{E}\left(\mathbf{1}_{\mathcal{B}} \cdot F_{f,\mathbb{X}}^k\right)+\mathbb{E}\left(\mathbf{1}_{\mathcal{B}^c} \cdot F_{f,\mathbb{X}}^k\right) \\
& =\mathbb{E}\left(\mathbf{1}_{\mathcal{B}} \cdot\left(F_{f,N, \mathbb{X}}+O\left(N^{-1 / 20}\right)^k\right)+\mathbb{E}\left(\mathbf{1}_{\mathcal{B}^c} \cdot F_{f,\mathbb{X}}^k\right)\right. \\
& =\mathbb{E}\left(\left(F_{f,N, \mathbb{X}}+O\left(N^{-1 / 20}\right)\right)^k\right)\\
&+O\left(\mid \mathbb{E}\left(\mathbf{1}_{\mathcal{B}^c} \cdot\left(F_{f,N, \mathbb{X}}+O\left(N^{-1 / 20}\right)^k\right)\left|+\left|\mathbb{E}\left(\mathbf{1}_{\mathcal{B}^c} \cdot F_{f,\mathbb{X}}^k\right)\right|\right.\right.\right) .
\end{aligned}
\]
 Applying the Binomial Theorem together with Lemma \ref{le6.1}, we obtain
\[
\mathbb{E}\left(F_{f,N, \mathbb{X}}^k\right)+E_3,
\]
where
\[
\begin{aligned}
    E_3 &\ll \sum_{j=1}^k\binom{k}{j}\left(c_{f,4} N\right)^{-j / 20} \mathbb{E}\left(\left|F_{f,N, \mathrm{X}}\right|^{k-j}\right)\\
    &\ll N^{-1 / 20}\left(c_{f,5} k^{2 / 3}\left(\log 2k\right)^{2/3}\right)^k \ll N^{-1 / 40},
\end{aligned}
\]
for some positive constant $c_{f,4}$ and $c_{f,5}$. Next, by the Cauchy--Schwarz inequality, Lemma \ref{le6.1} and  \eqref{eq6.14} we deduce that
\[
\begin{aligned}
& \mid \mathbb{E}\left(\boldsymbol{1}_{\mathcal {B}^{c}}\cdot (F_{f, N,\mathbb{X}}+O(N^{-1/20} )^{k} ) \left|+\left|\mathbb{E}\left(\mathbf{1}_{\mathcal{B}^c} \cdot F_{f,\mathbb{X}}^k\right)\right|\right.\right. \\
& \leqslant \mathbb{P}\left(\mathcal{B}^c\right)^{1 / 2}\left(\mathbb{E}\left(\left|F_{f,N, \mathbb{X}}+O\left(N^{-1/20}\right)\right|^{2 k}\right)^{1 / 2}+\left|\mathbb{E}\left(\left|F_{f,\mathbb{X}}\right|^{2 k}\right)\right|^{1 / 2}\right) \\
& \ll N^{-1 / 10}\left(2c_{f,1} k^{2 / 3}\left(\log 2k\right)^{2/3}\right)^k \ll N^{-1 / 20} .
\end{aligned}
\]
Collecting the above estimates, we obtain
\[
\mathbb{E}\left(F_{f,N, \mathbb{X}}^k\right)=\mathbb{E}\left(F_{f,\mathbb{X}}^k\right)+O\left(N^{-1 / 40}\right) .
\]
Substituting the above into \eqref{eq6.13} yields
\[
\begin{aligned}
\frac{1}{T} \int_{\mathcal{A}} \exp (\lambda F_f(t)) d t & =\sum_{k=0}^{2 L} \frac{\lambda^k}{k!} \mathbb{E}\left(F_{f,\mathbb{X}}^k\right)+O\left(N^{-1 / 40} e^{|\lambda|}+e^{-L}\right) \\
& =\mathbb{E}\left(e^{\lambda F_{f,\mathbb{X}}}\right)+E_4,
\end{aligned}
\]
where
\[
E_4 \ll e^{-L}+\sum_{k>2 L} \frac{|\lambda|^k}{k!} \mathbb{E}\left(\left|F_{f,\mathbb{X}}\right|^k\right) \ll e^{-L}+\sum_{k>2 L}\left(\frac{3 c_{f,1}|\lambda|\left(\log 2k\right)^{2/3}}{k^{1 /3}}\right)^k,
\]
by Stirling's formula and Lemma \ref{le6.1}. Finally, our assumption on $\lambda$ ensures that for the tail of the series,
\[
\sum_{k > 2L} \left( \frac{3 c_{f,1} |\lambda|\left(\log 2k\right)^{2/3}}{k^{1 /3}} \right)^k
\ll \sum_{k > 2L} \left( \frac{3 c_{f,1} |\lambda|\left(\log 4L\right)^{2/3}}{L^{1 /3}} \right)^k
\ll e^{-L}.
\]
 This completes the proof.
\end{proof}

Now we can present the proof of Theorem \ref{th1.4}. 
\begin{proof}
For a real number $\alpha$ we define
\[
\varphi_{F_f, T}(\alpha):=\frac{1}{T} \int_T^{2 T} e^{i \alpha F_f(t)} d t \text { and } \varphi_{F_{f,\mathrm{X}}}(\alpha):=\mathbb{E}\left(e^{i \alpha F_{f,\mathbb{X}}}\right).
\]
 Then it follows from Lemma \ref{le6.5} that uniformly for $\alpha$ in the range $|\alpha| \leqslant c_{f,0}\frac{\left(\log\log T\right)^{1/3}}{\left(\log\log\log T\right)^{5/3}}$ we have
\begin{equation}\label{eq6.15}
\begin{aligned}
    \varphi_{F_f, T}(\alpha)&=\frac{1}{T} \int_{\mathcal{A}} e^{i \alpha F_f(t)} d t+O\left(\frac{\operatorname{meas}\left(\mathcal{A}^c\right)}{T}\right)\\
    &=\varphi_{F_{f,\mathbb{X}}}(\alpha)+O\left(\exp \left(-\frac{\log \log  T}{\log \log \log T}\right)\right),
\end{aligned}
\end{equation}

Let $R:=c_{f,0}\frac{\left(\log\log T\right)^{1/3}}{\left(\log\log\log T\right)^{5/3}}$. By the Berry--Esseen inequality (see \cite[Lemma 1.47]{MR551361}) we have
\begin{equation}\label{eq6.16}
    \sup _{u \in \mathbb{R}}\left|\mathbb{P}_T(F_f(t) \leqslant u)-\mathbb{P}\left(F_{f,\mathbb{X}} \leqslant u\right)\right| \ll \frac{1}{R}+\int_{-R}^R\left|\frac{\varphi_{F_f, T}(\alpha)-\varphi_{F_{f,\mathbb{X}}}(\alpha)}{\alpha}\right| d \alpha.
\end{equation}
In the range $1 / \log T \leqslant|\alpha| \leqslant R$ we use \eqref{eq6.15} which gives
\begin{equation}\label{eq6.17}
    \int_{1 / \log T \leqslant|\alpha| \leqslant R}\left|\frac{\varphi_{F_f, T}(\alpha)-\varphi_{F_{f,\mathbb{X}}}(\alpha)}{\alpha}\right| d \alpha \ll \exp \left(-\frac{\log \log T}{2 \log \log\log T}\right).
\end{equation}
We now treat the remaining range $0 \le |\alpha| \le 1/\log T$. Applying the inequality $|e^{iv}-1| \ll |v|$, valid for all real $v$, we obtain
\[
\varphi_{F_{f,\mathbb{X}}}(\alpha)=1+O\left(|\alpha| \mathbb{E}\left(\left|F_{f,\mathbb{X}}\right|\right)\right)=1+O(|\alpha|),
\]
and similarly
\[
\varphi_{F_f, T}(\alpha)=1+O\left(|\alpha| \frac{1}{T} \int_T^{2 T}|F_f(t)| d t\right)=1+O(|\alpha|),
\]
where the last estimate follows from \eqref{eq5.1}. Therefore we deduce that
\[
\int_{-1 / \log T}^{1 / \log T}\left|\frac{\varphi_{F_f, T}(\alpha)-\varphi_{F_{f,\mathrm{X}}}(\alpha)}{\alpha}\right| d \alpha \ll \int_{-1 / \log T}^{1 / \log T} 1 d \alpha \ll \frac{1}{\log T} .
\]
Combining this bound with \eqref{eq6.16} and \eqref{eq6.17} completes the proof.
\end{proof}
\section{The Proof of Theorem \ref{th1.5}}
\begin{lemma}\label{le7.1}
     For a fixed self-dual Hecke--Maass cusp form $f$ of $\operatorname{GL}(3)$. There  exist positive constant $c_{f,6}$, $c_{f,7}$ such that for every real  number $\lambda$ with $|\lambda|>2$ and any positive $\varepsilon$, we have 
    \[
    \exp(c_{f,6}\lambda^{8/7})\leqslant \mathbb{E}(e^{\lambda F_{f,\mathbb{X}}})\leqslant\exp(c_{f,7}\lambda^{5/2+\varepsilon}).
    \]
\end{lemma}
\begin{proof}
    Since $X = \sum_{n=1}^\infty a_n(t_n)$ and the $t_n$'s are independent random variables uniformly distributed on $[0,1]$, we have, for $\lambda \geq 1$,
    \[
    E\bigl(\exp(\pm \lambda X)\bigr)=\prod_{n=1}^\infty E\bigl(\exp(\pm \lambda a_n(t))\bigr)=\prod_{n=1}^\infty\int_0^1 \exp\bigl(\pm \lambda a_n(t)\bigr)\,dt .
    \]
    First, we observe that $e^x\leq 1+x+x^2$ for $x\leq 1$; $e^x\geq 1+x$ for all $x$; and $e^x\geq 1+x+\frac{x^2}{2}+\frac{x^3}{6}$ for all $x$. Moreover, we have $|a_n(t)|\leq c_f n^{-\frac{5}{12}+\varepsilon}$.
    
    Let $\varepsilon_0>0$ be sufficiently small.  For those $n$ with $\lambda n^{-\frac{7}{18}}<\varepsilon_0$, using $\int_0^1 a_n(t)\,dt=0$ we obtain
    \[
    \begin{aligned}
    &1+\frac{\lambda^2}{4}\int_0^1 a_n(t)^2\,dt\\
    &\leq 1\pm\lambda\int_{0}^{1}a_{n}(t)\,dt+\frac{\lambda^{2}}{2}\int_{0}^{1}a_{n}(t)^{2}\,dt\pm\frac{\lambda^{3}}{6}\int_{0}^{1}a_{n}(t)^{3}\,dt \\
    &\leq \int_{0}^{1}\exp\bigl(\pm\lambda a_{n}(t)\bigr)\,dt \leq 1+\lambda^{2}\int_{0}^{1}a_{n}(t)^{2}\,dt .
    \end{aligned}
    \]
    
    Consequently, since $\log(1+y)\leq y$ for $y\geq 0$ and $\log(1+y)\geq y/2$ when $0\leq y\leq 1$, we deduce
    \[
    \frac{\lambda^{2}}{8}\int_{0}^{1}a_{n}(t)^{2}\,dt\;\leq\
    \log\int_{0}^{1}\exp\bigl(\pm\lambda a_{n}(t)\bigr)\,dt\;\leq\;
    \lambda^{2}\int_{0}^{1}a_{n}(t)^{2}\,dt .
    \]
    
    If $\lambda n^{-\frac{7}{18}}\geq \varepsilon_0$, then
    \[
    1=\int_0^1\bigl(1\pm\lambda a_n(t)\bigr)\,dt\;\leq\;
    \int_0^1\exp\bigl(\pm\lambda a_n(t)\bigr)\,dt\;\leq\;
    \exp\bigl(c_f\lambda n^{-\frac{5}{12}+\varepsilon}\bigr).
    \]
    
    Hence,
    \[
    \lambda^{8/7}\ll\lambda^{2}\sum_{\substack{\lambda n^{-\frac{7}{18}}<\varepsilon_0}}
    \sum_{r=1}^{\infty}\varepsilon(n)^{2}\frac{A_f(1,nr^{3})^{2}}{(nr^{3})^{4/3}}
    \ll\log E\bigl(\exp(\pm\lambda X)\bigr),
    \]
    and 
    \[
    \begin{aligned}
    \log E\bigl(\exp(\pm\lambda X)\bigr)
    &\ll\lambda^{2}\sum_{\substack{\lambda n^{-\frac{7}{18}}<\varepsilon_0}}
    \sum_{r=1}^{\infty}\varepsilon(n)^{2}\frac{A_f(1,nr^{3})^{2}}{(nr^{3})^{4/3}}
    +\lambda\sum_{\substack{\lambda n^{-\frac{7}{18}}\geq\varepsilon_0}} n^{-\frac{5}{12}+\varepsilon} \\
    &\ll \lambda^{5/2+\varepsilon}.
    \end{aligned}
    \]
    This completes the proof.
\end{proof}
Now we can present the proof of Theorem \ref{th1.5}. 
\begin{proof}
    We only prove the upper bound for $\mathbb{P}_T(F_f(t)>V)$, as the lower bound for $\mathbb{P}_T(F_f(t)<-V)$ follows by the same argument applied to $-F_f(t)$.  
    
    Let $\mathcal{A}$ be the set appearing in Lemma~\ref{le6.5} and choose a parameter $\lambda$ satisfying  
    $0<\lambda\leq c_{f,0}(\log\log T)^{1/3}(\log\log\log  T)^{-5/3}$.  
    By Lemma~\ref{le6.5},
    \[
    \begin{aligned}
    \mathbb{P}_{T}\bigl(F_f(t)>V\bigr)
    &\leq\frac{1}{T}\,\mathrm{meas}(\mathcal{A}^{c})
    +\frac{1}{T}\int_{\mathcal{A}}\exp\!\Bigl(\lambda\bigl(F_f(t)-V\bigr)\Bigr)\,dt \\
    &\ll\frac{1}{(\log T)^{20}}+e^{-\lambda V}\;
    \mathbb{E}\Bigl(e^{\lambda F_{f,\mathbb{X}}}\Bigr).
    \end{aligned}
    \]
    Using Lemma~\ref{le7.1} and selecting $\lambda$ so that $V=2c_{f,7}\lambda^{3/2+\varepsilon}$ yields
    \[
    \mathbb{P}_T\bigl(F_f(t)>V\bigr)
    \ll\frac{1}{(\log T)^{20}}+e^{-\lambda V/2}
    \ll\exp\Bigl(-c_{f,8}\,V^{5/3-\varepsilon}\Bigr),
    \]
    for a suitable positive constant $c_{f,8}$.
    
    We now establish the lower bound.  
    Set $\lambda_{\max}=c_{f,0}(\log\log T)^{1/3}(\log\log\log  T)^{-5/3}$ and let $\lambda$ be the number  satisfying
    \begin{equation}\label{eq7.1}
    \frac{1}{\operatorname{meas}(\mathcal{A})}
    \int_{\mathcal{A}}e^{\lambda F_f(t)}\,dt = 2e^{\lambda V}.
    \end{equation}
    Such a $\lambda$ exists and lies in the interval $(0,\lambda_{\max})$.  
    Indeed, both sides of \eqref{eq7.1} are continuous in $\lambda$; the left–hand side is smaller when $\lambda=0$, while for $\lambda=\lambda_{\max}$ Lemma~\ref{le6.5} and Lemma~\ref{le7.1} give
    \[
    \begin{aligned}
    \frac{1}{\operatorname{meas}(\mathcal{A})}\int_{\mathcal{A}} e^{\lambda_{\max} F_f(t)}\,dt
    &=(1+o(1))\;\mathbb{E}\Bigl(e^{\lambda_{\max} F_{f,\mathbb{X}}}\Bigr)+o(1) \\
    &\geq\exp\Bigl(\frac{c_{f,6}}{2}\lambda_{\max}^{8/7}\Bigr)
    >2e^{\lambda_{\max}V},
    \end{aligned}
    \]
    the last inequality following from our hypothesis on $V$ provided the constant $b_{f,2}$ is chosen sufficiently small.
    
    Combining Lemma~\ref{le6.5}, Lemma~\ref{le7.1} and \eqref{eq7.1} we obtain
    \[
    \lambda V = \log\!\Bigl(\mathbb{E}\bigl(e^{\lambda F_{f,\mathbb{X}}}\bigr)\Bigr)+O(1)
    \gg \lambda^{8/7},
    \]
    hence
    \[
    \lambda \ll  V^{7}.
    \]
    Now apply the Paley–-Zygmund inequality together with \eqref{eq7.1}:
    \[
    \begin{aligned}
    \mathbb{P}_T\bigl(F_f(t) > V\bigr)
    &=\frac{1}{\operatorname{meas}(\mathcal{A})}\;
    \operatorname{meas}\!\Bigl\{t\in\mathcal{A}:e^{\lambda F_f(t)}>e^{\lambda V}\Bigr\}
    +O\!\Bigl(\frac{1}{(\log T)^{20}}\Bigr) \\[2mm]
    &\geq\frac{1}{4}\,
    \frac{\displaystyle\Bigl(\frac{1}{\operatorname{meas}(\mathcal{A})}
          \int_{\mathcal{A}} e^{\lambda F_f(t)}\,dt\Bigr)^{\!2}}
        {\displaystyle\frac{1}{\operatorname{meas}(\mathcal{A})}
          \int_{\mathcal{A}} e^{2\lambda F_f(t)}\,dt}
          +O\!\Bigl(\frac{1}{(\log T)^{20}}\Bigr).
    \end{aligned}
    \]
    Using \eqref{eq7.1} again together with Lemma~\ref{le6.5} and Lemma~\ref{le7.1},
    \[
    \mathbb{P}_T\bigl(F_f(t) > V\bigr)
    \gg\frac{e^{2\lambda V}}{\mathbb{E}\bigl(e^{2\lambda F_{f,\mathbb{X}}}\bigr)}
    +\frac{1}{(\log T)^{20}}
    \gg\exp\!\Bigl(-c_{f,9}\,\lambda^{5/2+\varepsilon}\Bigr),
    \]
    for some positive constant $c_{f,9}$; here we used the trivial estimate $e^{\lambda V}\geq1$.  
    Recalling that $\lambda\ll V^{7}$ finishes the proof.
\end{proof}

\subsection*{Acknowledgements}
The author is most grateful to Prof. Yongxiao Lin for suggesting the problem and for all his guidance throughout the manuscript. The author would also like to express his thanks to Prof. Bingrong Huang and Prof. Ze\'ev Rudnick for their encouragement and useful suggestions which enriched his understanding. The author would additionally like to thank Prof. Yuk-Kam Lau for an enlightening talk at Shandong University in 2025 and Prof. Youness Lamzouri for his insights and suggestions regarding Theorem \ref{th1.5}.

\bibliographystyle{plain}
\bibliography{name}

\end{document}